\documentclass[11pt,a4paper]{amsart}
\usepackage{amsmath}
\usepackage{amsfonts}
\usepackage{amssymb}
\usepackage{amsthm}
\usepackage{mathrsfs}
\usepackage{enumitem}
\usepackage{graphicx}
\usepackage{url}
\usepackage{xcolor}
\usepackage{mathtools}

\newtheorem{theorem}{Theorem}[section]
\newtheorem{proposition}[theorem]{Proposition}

\newtheorem{corollary}[theorem]{Corollary} 

\newtheorem{definition}[theorem]{Definition}

\date{}

\title[Spectral Barron spaces arising from quantum harmonic analysis]{Spectral Barron spaces arising from quantum harmonic analysis}
\author{Yaogan Mensah}

\address{Department of Mathematics, University of Lom\'e, Togo \\P.O.Box 1515 Lom\'e 1, Togo} 
\email{\textcolor[rgb]{0.00,0.00,0.84}{mensahyaogan2@gmail.com,  ymensah@univ-lome.tg}}
\begin{document}

\begin{abstract}
In this paper, spectral Barron spaces are defined in the framework of quantum harmonic analysis. Their fundamental properties are studied. These include, among others, their   completeness structure and some continuous embedding results. As an application, the existence and the uniqueness of the solution of a Schr\"odinger-type equation are proved.  
\end{abstract}
\maketitle


\section{Introduction}
 Barron \cite{Barron},  showed that feedforward networks with one layer of sigmoidal nonlinearities achieve integrated squared error of order $O(n^{-1})$, where $n$ is the number of nodes. The approximated function is assumed to have a bound on the first moment of the magnitude distribution of the Fourier transform. From then, a class of functions known as spectral Barron space emerged 
 as a meeting point of research fields such as machine learning,  functional analysis, partial differential equations, etc (see \cite{Chen, Choulli,Meng} and references therein).
 
 The condition that characterizes the functions of the classical spectral Barron spaces is related to the Fourier transformation. This justifies the fact that we can study spectral Barron spaces in any framework where such a transformation exists with, of course, the necessary properties. In \cite{Fulsche}, Fulsche and Galke extended the notions of quantum harmonic analysis to abelian phase spaces (locally compact abelian groups endowed with a Heisenberg multiplier) following  the work of Werner \cite{Werner}.  They defined the  quantum Fourier transform of trace-class operators and studied  its major  properties. Thus, a framework  is created for the construction and the  analysis of spectral Barron spaces consisting of bounded operators on a Hilbert space. This framework lies at the intersection of harmonic analysis and operator theory  (quantum harmonic analysis). 
 
 The aim of this article is to study spectral  Barron spaces in the light of quantum harmonic analysis by following the steps of the article \cite{Choulli} in which  spectral Barron spaces on $\mathbb{R}^d$ are investigated. We defined spectral Barron spaces whose elements are bounded linear operators on a Hilbert space. These spaces are equipped with a Banach space structure. Some isomorphims and continuous embeddings between the  spectral Barron spaces on one hand and  the  spectral Barron spaces and the corresponding Sobolev spaces on the other hand  are obtained. Some of the results are invested in solving a Schr\"odinger-type equation.  The particularity in our case lies in the fact that the data of the problem, especially the potential,  are operators of a spectral Barron space.
 
 The rest of the article is organised as follows.  Section \ref{Preliminaries} collects facts about  compact and Schatten $p$-classes operators, Pontrajagin dual of a locally compact group,  quantum Fourier transform and  Banach contraction principle.  In Section \ref{Spectral Barron spaces},    spectral Barron spaces related to the quantumFourier transform are defined and their properties are explored. In Section \ref{Schrodinger}, the theory is applied to explored the existence and the  uniqueness  of the solution of a Schr\"odinger-type equation.

\section{Preliminaries}\label{Preliminaries}

\subsection{Compact and Schatten $p$-classes operators}
Here, we recall information related to compact operators and Schatten class operators. We limit ourselves to the facts we need, and for other details, we refer \cite{Dell, Zhu}.
 In this subsection and throughout the rest of the article, we denote by $H$  a separable complex Hilbert space. Let $B_1$  denote the closed unit ball of $H$, that is $B_1=\{x\in H : \|x\|_H\leq 1\}$. A linear operator $T$ on $H$ is said to be compact if the image  of $B_1$ by $T$ has compact closure. We denote by $\mathcal{K}(H)$ the set of compact operators on $H$. 
It is proved that $\mathcal{K}(H)$ is a two-sided ideal in $\mathcal{B}(H)$, the space of all bounded linear  operators on $H$, equipped with the uniform norm
\begin{equation}
\|T\|_{\mbox{op}}=\sup\left\{\|Tx\|_H : \|x\|_H\leq 1 \right\}.
\end{equation} 
The set $\mathcal{K}(H)$ inherits the topology of $\mathcal{B}(H)$.

If $T$ is a compact operator, then we can find two orthonormal  sets  $\{e_n\}$  and  $\{\varepsilon_n\}$ in $H$ such that for $x\in H$, 
\begin{equation}
Tx=\sum_{n=1}^\infty\lambda_n \langle x, e_n\rangle \varepsilon_n
\end{equation}
where  $\{\lambda_1, \lambda_2,\cdots\}$ is  the set of  singular values of $T$ nonincreasing ordered. 

For $1\leq p<\infty$, the Schatten $p$-class $S^p(H)$ is defined to be the space of compact operators on $H$ such that $\sum_{n=1}^\infty |\lambda_n|^p<\infty$. The  Schatten $p$-norm on $S^p(H)$ is given by :
\begin{equation}
\|T\|_{S^p(H)}=\left(\sum_{n=1}^\infty |\lambda_n|^p\right)^{\frac{1}{p}}. 
\end{equation}
and $S^p(H)$ is a Banach space under this norm. 
Moreover, if $p\leq q$, then $S^p(H)\subset S^q(H)$ and $\|T\|_{\mbox{op}}\leq \|T\|_{S^p(H)}\leq \|T\|_{S^q(H)}$. 
The two famous Schatten classes are the trace-class $S^1(H)$ and the Hilbert-Schmidt class $S^2(H)$.  Finally let us mention that :
\begin{enumerate}
\item[i)] $S^1(H)$ is dense in $S^2(H)$ in the $\|\cdot\|_{S^2(H)}$-topology,
\item[ii)] $S^2(H)$ is dense in $\mathcal{K}(H)$ in the uniform operator topology,
\item[iii)] $\mathcal{K}(H)$ is dense in $\mathcal{B}(H)$ in the strong operator topology.
\end{enumerate}

\subsection{Pontrjagin dual of a locally compact abelian group}
For more details about this subsection, we refer \cite{Hewitt, Rudin}.  

Let $G$ be a locally compact abelian  and Hausdorff group (we will write LCA group for the sake of brevity) with addition as multiplication.
Let  $\mathbb{T}$ denote the multiplicative group of complex numbers of modulus 1. 

A character of $G$ is any group homomorphism $\xi : G\rightarrow \mathbb{T}$. In other words for all $x,y\in G$, 
\begin{enumerate}
\item[i)] $|\xi(x)|=1$,
\item[ii)] $\xi(x+y)=\xi(x)\xi(y)$.
\end{enumerate}
 The set of all continuous characters of $G$ form an abelian group denoted by $\widehat{G}$ under the following law : 
 $$(\xi_1\xi_2)(x)=\xi_1(x)\xi_2(x), \xi_1, \xi_2\in \widehat{G}, x\in G.$$ 
 The inverse of $\xi \in \widehat{G}$ is given by 
 $$\xi^{-1}(x)=\overline{\xi (x)},\, x\in G.$$
The group  $\widehat{G}$ is topologized as follows.  
Let $K$ be a compact subset of $G$.  
 For $r>0$, let $U_r=\{ z\in \mathbb{C} : |1-z|<r\}$ and define $N(K,r)$ by 
 $$N(K,r)=\{\xi\in \widehat{G} : \xi (x)\in U_r, \forall x\in K\}.$$
The family of all sets $N(K,r)$ and their translates form a base of a topology with respect to  which  $\widehat{G}$  is a LCA  group \cite[Section 1.2.6]{Rudin}. 
Classical examples of dual groups are the following:
$$\widehat{\mathbb{T}}=\mathbb{Z},\,\widehat{\mathbb{R}}=\mathbb{R} \mbox{ and }\widehat{\mathbb{Z}}=\mathbb{T}.$$
Pontryagin duality theorem states that there exists a canonical isomorphism from $G$ onto $\widehat{\widehat{G}}$ (see \cite{Hewitt}). 

\subsection{Quantum Fourier transform}
In this subsection, let us briefly describe the quantum Fourier transform as formulated by Fulsche and Galke \cite{Fulsche}. 

Let $G$ be a LCA group with neutral element denoted by 0. A projective unitary representation of $G$ on the Hilbert space $H$ is a map $ \rho$ from $G$ into   $\mathcal{U}(H)$, the group of unitary operators of $H$, such that there exists  a function $m: G\times G \rightarrow \mathbb{T}$ satisfying 
\begin{equation}
\rho(x)\rho(y)=m(x,y)\rho(x+y), \forall x,y\in G.
\end{equation}
 The function $m$ is called the multiplier  of the projective representation $\rho$. For $x\in G$, we write $U_x$ in place of $\rho (x)$. It is assumed that $U_0=I$, the identity operator in $\mathcal{U}(H)$. Also, the function $m$ is assumed to be separately  continuous. Let us mention that the adjoint of  $U_x$ is given by 
 $U_x^*=\overline{m(x,-x)}U_{-x}$.
 
 Put $\sigma(x,y)=\displaystyle\frac{m(x,y)}{m(y,x)}, x,y\in G$. From the continuous property of $m$, it holds that the mapping $x\mapsto \sigma(x, \cdot)$ from $G$ into $\widehat{G}$ is a continuous homomorphism. The function $m$ is called a Heisenberg multiplier if the above mapping is a topological isomorphism. In this case, the group $G$ and its dual $\widehat{G}$ can be identified with each other. 
 
 The assumptions about the projective representation  behind the construction of the quantum Fourier transform   are as follows:
 \begin{enumerate}
 \item[i)] $\sigma$ is a Heisenberg multiplier,
 \item[ii)] $\rho$ is square integrable; that is, there exists nonzero vectors  $\varphi,\eta\in H$ such that the mapping $x\mapsto \langle U_x \varphi,\eta\rangle$ belongs to $L^2(G)$, 
 \item[iii)] $m$ satisfies : $m(x,y)=m(-x,-y), \forall x,y\in G$. 
 \end{enumerate}
  For some results, it is also assumed that $\rho$ is integrable, that is, there exists a nonzero vector $\varphi\in H$ such that  the mapping $x\mapsto \langle U_x \varphi,\varphi\rangle$ belongs to $L^1(G)$.

The quantum Fourier transform of the operator $T\in S^1(H)$ is defined as \cite{Fulsche} :
\begin{equation}
\mathcal{F}_U(T)(\xi)=\mbox{tr}(TU_\xi^*), \xi\in \widehat{G}, 
\end{equation}
where tr is the usual trace of a trace-class  operator and $U_\xi^*$ is the adjoint of the unitary  operator $U_\xi$. 

Moreover, if $T\in S^1(H)$ is such that $\mathcal{F}_U(T)\in L^1(\widehat{G})$, then the following inversion formula holds :
\begin{equation}
T=\int_{\widehat{G}}\mathcal{F}_U(T)(\xi)U_\xi d\xi.
\end{equation}

Now, we will state some results of \cite{Fulsche} that we have used in this article.
\begin{proposition}\cite[Proposition 6.26]{Fulsche}\label{Fulsche Proposition 6.26}
Assume the representation
is integrable. If $f\in L^1(\widehat{G})$, then $\mathcal{F}_U^{-1}(f)\in \mathcal{K}(H)$ and $\|\mathcal{F}_U^{-1}(f)\|_{\mbox{op}}\leq \|f\|_{L^1(\widehat{G})}.$
\end{proposition}
Let us recall that in this framework, the twisted convolution product of $f,g\in L^1(\widehat{G})$ is defined as
\begin{equation}
(f\ast_m g)(\xi)=\int_{\widehat{G}}f(\xi-\eta)g(\eta)m(\xi-\eta,\eta)d\eta.
\end{equation}
\begin{proposition}\cite[Corollary 6.24]{Fulsche}\label{Fulsche Corollary 6.24}
If $S,T\in S^1(H)$, then 
$$\mathcal{F}_U(ST)=\mathcal{F}_U(S) \ast_m\mathcal{F}_U(T).$$
\end{proposition}

\subsection{Banach contraction principle}
To end these preliminary notes, we recall a well-known result from the literature : Banach contraction principle. It is  almost indispensable in establishing the existence of a fixed point of a mapping. 

\begin{definition}
Let $(X,d)$ be a metric space. Let $f: X\rightarrow X$ be a map. 
\begin{enumerate}
\item A point $x$ in $X$ is called a fixed point of $f$ if $x=f(x)$.
\item The mapping $f$ is said to be  contraction if there exists $h<1$ such that $\forall x,y\in X$,
$$d(f(x),f(y))\leq hd(x,y).$$ 
\end{enumerate}
\end{definition}
\begin{theorem}(see \cite[Theorem 2.1]{Latif})\label{BCP}
Let $(X,d)$ be a complete metric
space. Then each contraction map $f : X \rightarrow  X$ has a unique fixed point.
\end{theorem}

\section{Spectral Barron spaces of compact operators: Properties and embeddings}\label{Spectral Barron spaces}

In this section, we define the spectral Barron spaces whose elements are bounded linear operators, thus bringing these spaces into the realm of operator theory for the first time, with the hope that they can be useful in quantum physics. 

Let $H$ be a complex separable Hilbert space, $\mathcal{B}(H)$ the space of bounded linear operators on $H$, $\widehat{G}$ the dual group of a localement compact group $G$. 
 Let $L^1(\widehat{G})$ denote the Lebesgue space of  functions which are integrable with respect to the Haar measure on $\widehat{G}$. 

Let $\gamma :\widehat{G}\rightarrow (0,\infty)$ be a measurable mapping and let $s$ be a nonnegative real number.  Hereafter is the main definition of this article.

\begin{definition}  We call spectral Barron space the set 
\begin{equation}
\mathfrak{B}^s_\gamma (H)=\left\{T\in \mathcal{B}(H): (1+\gamma(\xi)^2)^{\frac{s}{2}}\mathcal{F}_U(T)\in L^1(\widehat{G})\right\}
\end{equation}
 equipped with the  norm defined  by 

\begin{equation}
\|T\|_{\mathfrak{B}^s_\gamma (H)}=\left\|(1+\gamma(\xi)^2)^{\frac{s}{2}}\mathcal{F}_U(T)\right\|_{L^1(\widehat{G})}=\int_{\widehat{G}}(1+\gamma(\xi)^2)^{\frac{s}{2}}|\mathcal{F}_U(T)(\xi)|d\xi.
\end{equation}
\end{definition}
It can easily be observed that spectral Barron spaces are (complex) vector spaces.
The following result states the completeness of the spectral Barron space $\mathfrak{B}^0_\gamma(H)$.
\begin{theorem}\label{Banach space}
The spectral Barron space $\mathfrak{B}^0_\gamma(H)$ is a Banach space.  
\end{theorem}
\begin{proof}
Let $\{T_n\}$ be a Cauchy sequence in $\mathfrak{B}^0_\gamma(H)$. Since $\|T_n\|_{\mathfrak{B}^0_\gamma}=\|\mathcal{F}_U(T_n)\|_{L^1(\widehat{G})}$, then $\{\mathcal{F}_U(T_n)\}$ is a Cauchy sequence in $L^1(\widehat{G})$. However,    $L^1(\widehat{G})$ is a complete space. Therefore, there exists $\psi\in L^1(\widehat{G})$ such that $\|\mathcal{F}_U(T_n)-\psi\|_{L^1(\widehat{G})}$ converges to 0 as $n$ goes to $\infty$. Set $W=\mathcal{F}_U^{-1}(\psi)$. We have $\mathcal{F}_U(W)=\psi \in L^1(\widehat{G})$, thus $W\in \mathfrak{B}^0_\gamma(H)$. Finally, 
$$\lim\limits_{n\rightarrow\infty}\|T_n-W\|_{\mathfrak{B}^0_\gamma}=\lim\limits_{n\rightarrow\infty}\|\mathcal{F}_U(T_n)-\psi\|_{L^1(\widehat{G})}=0.$$
Thus, $\mathfrak{B}^0_\gamma (H)$ is complete.
\end{proof}
In the sequel, any linear map whose input and output are operators will be called a {\it transformer}.

Consider the transformer $\mathfrak{Q}_{\gamma,s}$  defined on the spectral Barron space $\mathfrak{B}^{2s}_\gamma (H)$ as follows : 
 
$$\mathfrak{Q}_{\gamma,s}T=\int_{\widehat{G}}(1+\gamma(\xi)^2)^s\mathcal{F}_U(T)(\xi)U_\xi d\xi=\mathcal{F}_U^{-1}\left[(1+\gamma(\xi)^2)^s\mathcal{F}_U(T)\right].$$
Then, $\mathcal{F}_U(\mathfrak{Q}_{\gamma,s}T)(\xi)=(1+\gamma(\xi)^2)^s\mathcal{F}_U(T)(\xi)$,
 and the latter equality  leads to 
$$\|\mathfrak{Q}_{\gamma,s}T\|_{\mathfrak{B}^0_\gamma (H)}=\|T\|_{\mathfrak{B}^{2s}_\gamma (H)}.$$ 

Moreover, let $S\in \mathfrak{B}^{0}_\gamma (H)$. Let $T$ be such that $\mathcal{F}_U(T)=(1+\gamma(\xi)^2)^{-s}\mathcal{F}_U(S)$. We have $T\in \mathfrak{B}^{2s}_\gamma (H)$ and $\mathfrak{Q}_{\gamma,s}T=S$. So, we have proved the following theorem. 

\begin{theorem}\label{Q isometry}
For each $s\geq 0$, the map $\mathfrak{Q}_{\gamma,s}$ is an isometric isomorphism from $\mathfrak{B}^{2s}_\gamma (H)$ onto $\mathfrak{B}^{0}_\gamma (H)$. 
\end{theorem}
We deduce the following consequence from Theorem \ref{Q isometry}.
\begin{corollary}
$\forall s\geq 0$, the spectral Barron space $\mathfrak{B}^{s}_\gamma (H)$ is a Banach space. 
\end{corollary}
\begin{proof}
 By Theorem \ref{Q isometry}, the space $\mathfrak{B}^{s}_\gamma (H)$ is isometrically isomorphic to $\mathfrak{B}^{0}_\gamma (H)$ (the isomorphism map being $\mathfrak{Q}_{\gamma,\frac{s}{2}}$) and by Theorem \ref{Banach space}, the space $\mathfrak{B}^{0}_\gamma (H)$ is complete. Therefore,  $\mathfrak{B}^{s}_\gamma (H)$ is complete. 
\end{proof}

Let us recall that $\mathcal{K}(H)$ is equipped with the uniform norm inherited from the space $\mathcal{B}(H)$ of bounded operators  on $H$.

\begin{theorem}
The transformer $\mathfrak{Q}_{\gamma,s}: \mathfrak{B}^{2s}_\gamma (H) \rightarrow \mathcal{K}(H)$ is bounded and $\|\mathfrak{Q}_{\gamma,s}\|\leq 1$.
\end{theorem}
\begin{proof}
Let $T\in \mathfrak{B}^{2s}_\gamma (H)$. Then, $(1+\gamma(\xi)^2)^{s}\mathcal{F}_U(T)\in L^1(\widehat{G})$ . Thus,  $\mathcal{F}_U^{-1}\left[(1+\gamma(\xi)^2)^{s}\mathcal{F}_U(T)\right]$ is a compact operator (Proposition \ref{Fulsche Proposition 6.26}) and 
$$\left\|\mathcal{F}_U^{-1}\left[(1+\gamma(\xi)^2)^s\mathcal{F}_U(T)\right]\right\|_{\mbox{op}}\leq \|(1+\gamma(\xi)^2)^s\mathcal{F}_U(T)\|_{L^1(\widehat{G})}.$$ 
That is, $\|\mathfrak{Q}_{\gamma,s}  T\|_{\mbox{op}}\leq \|T\|_{\mathfrak{B}^{2s}_\gamma (H)}$.

\end{proof}

In the following, the symbol $\hookrightarrow$ between two Banach spaces will mean that the one on the left embdes continuously  into the one on the right.

\begin{theorem}\label{continuous embedding}
\begin{enumerate}
\item[i)] If $0\leq s\leq t$, then $\mathfrak{B}^t_\gamma(H)\hookrightarrow \mathfrak{B}^s_\gamma(H)$ and 
$$\forall T\in \mathfrak{B}^t_\gamma(H),\, \|T\|_{\mathfrak{B}^s_\gamma(H)}\leq \|T\|_{\mathfrak{B}^t_\gamma(H)}.$$

\item[ii)] Let $0\leq r\leq t$,$s\in [r,t]$, let $\alpha \in [0,1]$ such that   $s=\alpha r + (1-\alpha)t$. Then,
$$\forall
T\in \mathfrak{B}^{t}_\gamma (H),\, \|T\|_{\mathfrak{B}^{s}_\gamma (H)}\leq \|T\|_{\mathfrak{B}^{r}_\gamma (H)}^\alpha \|T\|_{\mathfrak{B}^{t}_\gamma (H)}^{1-\alpha}.$$
\item[iii)] $\mathfrak{B}^0_\gamma(H)\hookrightarrow \mathcal{K}(H)$ and 
$\forall T\in  \mathfrak{B}^0_\gamma(H),\,\|T\|_{\mbox{op}}\leq \|T\|_{\mathfrak{B}^0_\gamma(H)}.$
\end{enumerate}
\end{theorem}

\begin{proof}
\begin{enumerate}
\item[i)] This is trivial when one notices that if $0\leq s\leq t$, then 
$$(1+\gamma(\xi)^2)^{\frac{s}{2}}\leq (1+\gamma(\xi)^2)^{\frac{t}{2}}.$$
\item[ii)] Let $T\in \mathfrak{B}^t_\gamma(H)$. Then $T\in \mathfrak{B}^s_\gamma(H)$ by i). We have 
\begin{align*}
\|T\|_{\mathfrak{B}^s_\gamma(H)}&=\int_{\widehat{G}}(1+\gamma(\xi)^2)^{\frac{s}{2}}|\mathcal{F}_U(T)(\xi)|d\xi\\
&=\int_{\widehat{G}}(1+\gamma(\xi)^2)^{\frac{\alpha r}{2}}|\mathcal{F}_U(T)(\xi)|^\alpha |\mathcal{F}_U(T)(\xi)|^{1-\alpha} (1+\gamma(\xi)^2)^{\frac{(1-\alpha)t}{2}}d\xi\\
&=\left(\int_{\widehat{G}}(1+\gamma(\xi)^2)^{\frac{r}{2}}|\mathcal{F}_U(T)(\xi)| d\xi\right)^\alpha \left(\int_{\widehat{G}}(1+\gamma(\xi)^2)^{\frac{t}{2}}|\mathcal{F}_U(T)(\xi)| d\xi\right)^{1-\alpha}     \\
&(\mbox{by the H\"older inequality})\\
&\leq \|T\|_{\mathfrak{B}^r_\gamma(H)}^\alpha \|T\|_{\mathfrak{B}^t_\gamma(H)}^{1-\alpha}.
\end{align*}
\item[iii)] Let $T\in \mathfrak{B}^0_\gamma (H)$. Then, $\mathcal{F}_U(T)\in L^1(\widehat{G})$. By Proposition \ref{Fulsche Proposition 6.26}, we have 
$\mathcal{F}_U^{-1}(\mathcal{F}_U(T))=T\in \mathcal{K}(H)$ and $\|T\|_{\mbox{op}}\leq \|\mathcal{F}_U(T)\|_{L^1(\widehat{G})}$; that is $\|T\|_{\mbox{op}}\leq \|T\|_{\mathfrak{B}^0_\gamma(H)}$.
\end{enumerate}
\end{proof}

Let us recall Peetre's inequality (see \cite{Barros}). Here $\|x\|$ designates the euclidean  norm of $x$ in $\mathbb{R}^n$.
\begin{theorem}
Let $x,y\in \mathbb{R}^n$ and let $s\in \mathbb{R}$. Then, 
$$\left(\frac{1+\|x\|^2}{1+\|y\|^2}\right)^s\leq 2^{|s|}(1+\|x-y\|^2)^{|s|}.$$
\end{theorem}
\vspace{1cm}

 The following result holds. It states the stability of the spectral Barron space $\mathfrak{B}^s_\gamma(H)$  with respect to the composition of operators. 

\begin{theorem}
Let $s\geq 0$. If $S,T\in \mathfrak{B}^s_\gamma(H)$, then $ST \in \mathfrak{B}^s_\gamma(H)$ and 
$$\|ST\|_{\mathfrak{B}^s_\gamma(H)}\leq 2^{\frac{s}{2}}\|S\|_{\mathfrak{B}^s_\gamma(H)} \|T\|_{\mathfrak{B}^{s}}.$$
\end{theorem}

\begin{proof}
Let $S,T\in \mathfrak{B}^s_\gamma(H)$. Then, we have 
\begin{align*}
|\mathcal{F}_U(ST)(\xi)|&=\left|\mathcal{F}_U(S) \ast_m\mathcal{F}_U(T)(\xi)\right|\\
&=\left|\int_{\widehat{G}}\mathcal{F}_U(S)(\xi-\eta)\mathcal{F}_U(T)(\eta)m(\xi-\eta,\eta)d\eta\right|\\
&\leq \int_{\widehat{G}}|\mathcal{F}_U(S)(\xi-\eta)||\mathcal{F}_U(T)(\eta)|d\eta. 
\end{align*}
Now, using Peetre's inequality, we attain
$$(1+\gamma(\xi)^2)^{\frac{s}{2}}|\mathcal{F}_U(ST)(\xi)|\leq (1+\gamma(\eta)^2)^{\frac{s}{2}} 2^{\frac{s}{2}}(1+\gamma(\xi-\eta)^2)^{\frac{s}{2}}\int_{\widehat{G}}|\mathcal{F}_U(S)(\xi-\eta)||\mathcal{F}_U(T)(\eta)|d\eta.$$
Therefore, 
\begin{align*}
\int_{\widehat{G}}(1+\gamma(\xi)^2)^{\frac{s}{2}}|\mathcal{F}_U(ST)(\xi)|d\xi &\leq 2^{\frac{s}{2}}\int_{\widehat{G}}\int_{\widehat{G}}(1+\gamma(\eta)^2)^{\frac{s}{2}} (1+\gamma(\xi-\eta)^2)^{\frac{s}{2}}\times \\
&\times |\mathcal{F}_U(S)(\xi-\eta)||\mathcal{F}_U(T)(\eta)|d\eta d\xi.
\end{align*}
Now, using succesively Fubini's Theorem and the invariance by translation of the Haar measure on $\widehat{G}$, we obtain
\begin{align*}
\int_{\widehat{G}}(1+\gamma(\xi)^2)^{\frac{s}{2}}|\mathcal{F}_U(ST)(\xi)|d\xi &\leq 2^{\frac{s}{2}}\int_{\widehat{G}}(1+\gamma(\xi)^2)^{\frac{s}{2}}|\mathcal{F}_U(S)(\xi)|d\xi\times\\
&\times \int_{\widehat{G}}(1+\gamma(\eta)^2)^{\frac{s}{2}}|\mathcal{F}_U(S)(\eta)|d\eta.
\end{align*}
That is, $ \|ST\|_{\mathfrak{B}^s_\gamma(H)}\leq 2^{\frac{s}{2}}\|S\|_{\mathfrak{B}^s_\gamma(H)} \|T\|_{\mathfrak{B}^s_\gamma(H)}$.
\end{proof}

Let us consider the (quantum) Sobolev space $$\mathfrak{H}^s_\gamma=\left\{T\in S^2(H) : (1+\gamma(\xi)^2)^{\frac{s}{2}}\mathcal{F}_U(T)\in L^2(\widehat{G})\right\}.$$   This space is equipped with the norm :
\begin{equation}
\|T\|_{\mathfrak{H}^s_\gamma}=\left( \int_{\widehat{G}}(1+\gamma(\xi)^2)^s|\mathcal{F}_U(T)(\xi)|^2d\xi \right)^{\frac{1}{2}}.
\end{equation} 

Under an additional condition, we prove in the following theorem that the  Sobolev space $\mathfrak{H}^s_\gamma$ embeds continuously into the spectral Barron space $\mathfrak{B}^s_\gamma(H)$. 

\begin{theorem} 
Let $t>s \geq 0$. If $(1+\gamma(\cdot)^2)^{\frac{s-t}{2}}\in L^2(\widehat{G})$, then 
$\mathfrak{H}^t_\gamma\hookrightarrow \mathfrak{B}^s_\gamma(H)$ and 
$$\forall T\in \mathfrak{H}^t_\gamma,\,  \|T\|_{\mathfrak{B}^s_\gamma(H)}\leq \|(1+\gamma(\cdot)^2)^{\frac{s-t}{2}}\|_{L^2(\widehat{G})}\|T\|_{\mathfrak{H}^t_\gamma}.$$
\end{theorem}

\begin{proof}
\begin{align*}
\|T\|_{\mathfrak{B}^s_\gamma(H)}&=\int_{\widehat{G}}(1+\gamma(\xi)^2)^{\frac{s}{2}}|\mathcal{F}_U(T)(\xi)|d\xi\\
&=\int_{\widehat{G}}(1+\gamma(\xi)^2)^{\frac{s-t}{2}}(1+\gamma(\xi)^2)^{\frac{t}{2}}|\mathcal{F}_U(T)(\xi)|d\xi\\
&\leq \left(\int_{\widehat{G}}(1+\gamma(\xi)^2)^{s-t}d\xi\right)^{\frac{1}{2}}\left(\int_{\widehat{G}}(1+\gamma(\xi)^2)^{t}|\mathcal{F}_U(T)(\xi)|^2 d\xi\right)^{\frac{1}{2}}\\
&(\mbox{by the H\"older inequality.})\\
&=\|(1+\gamma(\cdot)^2)^{\frac{s-t}{2}}\|_{L^2(\widehat{G})}\|T\|_{\mathfrak{H}^t_\gamma}.
\end{align*}
\end{proof}

Let $\alpha$ be a positive real number. Put $\mathfrak{Q}_{\gamma,\alpha}=\alpha \mathfrak{I}-\Delta$, where $\mathfrak{I}$ is the identity transformer and $\Delta$ is the Laplacian defined by $\Delta T=-\mathcal{F}_U^{-1}[\gamma (\xi)^2\mathcal{F}_U(T)]$. 
 Then, 
\begin{equation}
\mathfrak{Q}_{\gamma,\alpha}T=\mathcal{F}_U^{-1}[(\alpha +\gamma (\xi)^2)\mathcal{F}_U(T)].
\end{equation}
Clearly, the transformer $\mathfrak{Q}_{\gamma,\alpha}$ is injective.  We have the following result. 
\begin{theorem} Let $s\geq 0$.
If $T\in \mathfrak{B}^s_\gamma(H)$, then 
$\|\mathfrak{Q}_{\gamma,\alpha}^{-1}T\|_{\mathfrak{B}^s_\gamma(H)}\leq \displaystyle\frac{1}{\alpha}\|T\|_{\mathfrak{B}^s_\gamma(H)}$.
\end{theorem}
\begin{proof}
Soit $T\in \mathfrak{B}^s_\gamma(H)$. Set $S=\mathfrak{Q}_{\gamma,\alpha}^{-1}T$. Then, $\mathcal{F}_U(T)=(\alpha +\gamma (\xi)^2)\mathcal{F}_U(S)$.
\begin{align*}
\|S\|_{\mathfrak{B}^s_\gamma(H)}&=\|(1 +\gamma (\xi)^2)^{\frac{s}{2}}\mathcal{F}_U(S)\|_{L^1(\widehat{G})}\\
&=\frac{1}{\alpha +\gamma (\xi)^2}\|(1 +\gamma (\xi)^2)^{\frac{s}{2}}\mathcal{F}_U(T)\|_{L^1(\widehat{G})}\\
&\leq \frac{1}{\alpha}\|(1 +\gamma (\xi)^2)^{\frac{s}{2}}\mathcal{F}_U(T)\|_{L^1(\widehat{G})}\\
&=\frac{1}{\alpha}\|T\|_{\mathfrak{B}^s_\gamma(H)}.
\end{align*}
\end{proof}

\section{Application : Schr\"odinger-type equations for operators}\label{Schrodinger}
The Schr\"odinger equation is a cornerstone of quantum mechanics and it governs the microscopic world. In our framework, we consider the Schr\"odinger-type equation 
\begin{equation}\label{Schrodinger-type equation}
(\mathfrak{I}-\Delta +V)S=T
\end{equation}
where the potential $V$ and the given operator $T$ are both in the spectral Barron space $\mathfrak{B}_\gamma^{0}(H)$. The operator $S$ is the unknown. 
The equation (\ref{Schrodinger-type equation}) can be written using the transformer $\mathfrak{Q}_{\gamma,1}$ as :
\begin{equation}\label{Schrodinger-type equation 2}
(\mathfrak{Q}_{\gamma,1} +V)S=T.
\end{equation}
We prove the existence of a unique solution to the equation (\ref{Schrodinger-type equation}) for potentials belonging to the open unit ball of $\mathfrak{B}_\gamma^{0}(H)$. 
Put 
$$\mathfrak{U}_0=\{T\in \mathfrak{B}_\gamma^{0}(H) : \|T\|_{\mathfrak{B}_\gamma^{0}(H)}<1\}.$$

\begin{theorem}
Let $V\in \mathfrak{U}_0$ and $T\in \mathfrak{B}_\gamma^{0}(H)$. The  Schr\"odinger-type equation (\ref{Schrodinger-type equation}) has a unique solution $S_*\in \mathfrak{B}_\gamma^{2}(H)$. Moreover,
\begin{equation}
\|S_*\|_{\mathfrak{B}_\gamma^{2}(H)}\leq (1-\|V\|_{\mathfrak{B}_\gamma^{0}(H)})^{-1}\|T\|_{\mathfrak{B}_\gamma^{0}(H)}.
\end{equation} 
\end{theorem}
\begin{proof}
Assume that $V\in \mathfrak{U}_0$ and $T\in \mathfrak{B}_\gamma^{0}(H)$. 
Rewrite  the equation (\ref{Schrodinger-type equation 2}) as 
\begin{align*}
\mathfrak{Q}_{\gamma,1}S= -VS+T.
\end{align*}
Since $\mathfrak{Q}_{\gamma,1}$ is an isomorphism (from $\mathfrak{B}_\gamma^{2}(H)$ onto $\mathfrak{B}_\gamma^{0}(H)$), then  $\mathfrak{Q}_{\gamma,1}$ is invertible. Furthermore,   $\mathfrak{Q}_{\gamma,1}^{-1}$ is also an isomorphism (from $\mathfrak{B}_\gamma^{0}(H)$ onto $\mathfrak{B}_\gamma^{2}(H)$).  
We have 
$$S=\mathfrak{Q}_{\gamma,1}^{-1}(-VS+T).$$
Let us consider the map $\mathfrak{E}$ defined from $\mathfrak{B}_\gamma^{0}(H)$ into itself by
\begin{equation}
\mathfrak{E}S=\mathfrak{Q}_{\gamma,1}^{-1}(-VS+T).
\end{equation}
Let $X,Y\in \mathfrak{B}_\gamma^{0}(H)$. We have
\begin{align*}
\|\mathfrak{E}X-\mathfrak{E}Y\|_{\mathfrak{B}_\gamma^{0}(H)}&\leq \|V\|_{\mathfrak{B}_\gamma^{0}(H)}\|X-Y\|_{\mathfrak{B}_\gamma^{0}(H)}\\
&<\|X-Y\|_{\mathfrak{B}_\gamma^{0}(H)}\\
&(\mbox{because } \|V\|_{\mathfrak{B}_\gamma^{0}(H)}<1).
\end{align*}

Therefore,  $\mathfrak{E}$ is a contraction mapping on $\mathfrak{B}_\gamma^{0}(H)$. Now, from  the Banach contraction principle (Theorem \ref{BCP}), we deduce that $\mathfrak{E}$ has a fixed point, that is, there exists an operator $S_*\in \mathfrak{B}_\gamma^{0}(H)$ such that $\mathfrak{E}S_*=S_*$. From Theorem \ref{Q isometry}, we draw the fact that $S_*\in \mathfrak{B}_\gamma^{2}(H)$ as an inverse image by $\mathfrak{Q}_{\gamma,1}$  of an element of $\mathfrak{B}_\gamma^{0}(H)$. 

Moreover,
\begin{align*}
\|\mathfrak{Q}_{\gamma,1}S_*\|_{\mathfrak{B}_\gamma^{0}(H)}&=\|-VS_*+T\|_{\mathfrak{B}_\gamma^{0}(H)}\\
&\leq \|V\|_{\mathfrak{B}_\gamma^{0}(H)}\|S_*\|_{\mathfrak{B}_\gamma^{0}(H)}+ \|T\|_{\mathfrak{B}_\gamma^{0}(H)}.
\end{align*} 
However, we know that $\|\mathfrak{Q}_{\gamma,1}S_*\|_{\mathfrak{B}_\gamma^{0}(H)}=\|S_*\|_{\mathfrak{B}_\gamma^{2}(H)}$. Also $\|S_*\|_{\mathfrak{B}_\gamma^{0}(H)}\leq \|S_*\|_{\mathfrak{B}_\gamma^{2}(H)}$ (refer to Theorem \ref{continuous embedding} i)).  By cross-referencing all this information, we come to the fact that :
\begin{align*}
\|S_*\|_{\mathfrak{B}_\gamma^{2}(H)}&\leq \|V\|_{\mathfrak{B}_\gamma^{0}(H)}\|S_*\|_{\mathfrak{B}_\gamma^{2}(H)}+ \|T\|_{\mathfrak{B}_\gamma^{0}(H)}.
\end{align*}
This  implies 
$$\|S_*\|_{\mathfrak{B}_\gamma^{2}(H)}\leq (1-\|V\|_{\mathfrak{B}_\gamma^{0}(H)})^{-1}\|T\|_{\mathfrak{B}_\gamma^{0}(H)}.$$
\end{proof}

\section{Conclusion}
In this paper, spectral Barron spaces are defined and  their properties are studied. A Schr\"odinger-type in this framework is investigated. 
A lot of work can be done to integrate spectral Barron spaces as objects of study in abstract harmonic analysis. For example, we can study them on abelian groups, compact groups, Lie groups of various kind, etc. It would also be interesting  to study  duality among spectral Barron spaces within the framework of abstract groups.

\end{document}